\documentclass{article}
\usepackage{geometry}
\usepackage{amsmath,amssymb,amsthm,mathtools}
\usepackage{microtype}
\usepackage{booktabs,array,longtable}
\usepackage{enumitem}
\usepackage{graphicx}
\usepackage{xcolor}
\usepackage{tikz}
\usetikzlibrary{arrows.meta,positioning,calc}
\usepackage{hyperref}
\usepackage[nameinlink,capitalise,noabbrev]{cleveref}
\usepackage{url}
\usepackage{authblk}
\usepackage{setspace}
\hypersetup{colorlinks=true,linkcolor=blue!55!black,citecolor=blue!55!black,urlcolor=blue!55!black}
\setlist[itemize]{leftmargin=2em,itemsep=2pt,topsep=4pt}
\setlist[enumerate]{leftmargin=2.2em,itemsep=2pt,topsep=4pt}

\newtheorem{theorem}{Theorem}[section]
\newtheorem{proposition}[theorem]{Proposition}
\newtheorem{corollary}[theorem]{Corollary}

\theoremstyle{definition}
\newtheorem{definition}[theorem]{Definition}
\newtheorem{example}[theorem]{Example}
\newtheorem{algorithm}[theorem]{Algorithm}
\theoremstyle{remark}
\newtheorem{remark}[theorem]{Remark}

\newcommand{\Pset}{\mathcal P}
\newcommand{\taua}{\tau_{\mathfrak a}}

\newcommand{\cla}{\operatorname{cl}_{\mathfrak a}}
\newcommand{\clainf}{\operatorname{cl}_{\mathfrak a}^{\infty}}
\newcommand{\inta}{\operatorname{int}_{\mathfrak a}}
\newcommand{\Sa}{S_{\mathfrak a}}
\newcommand{\Sainf}{S_{\mathfrak a}^{\infty}}
\newcommand{\Ka}{K_{\mathfrak a}}
\newcommand{\preca}{\preccurlyeq_{\mathfrak a}}
\newcommand{\eqa}{\equiv_{\mathfrak a}}
\newcommand{\LQ}{\mathsf L_{\mathfrak a}(X)}
\newcommand{\upar}{\mathord{\uparrow}_{\mathfrak a}}
\newcommand{\downar}{\mathord{\downarrow}_{\mathfrak a}}

\newcommand{\Up}{\operatorname{Up}}

\title{\textbf{Levine Equivalence in Aura Topological Spaces:}\\
Reachability, Alexandrov Structure, and Quotient Frames}
\author[]{S. M. Elsayed$^{1}$ and Tareq M. Al-shami$^{2}$}
\affil[]{$^{1}$Department of Mathematics, Faculty of Science, Aswan University, Egypt\\\texttt{profsalah55@aswu.edu.eg}\\
$^{2}$Department of Mathematics, Sana'a University, Sana'a, Yemen\\\texttt{tareqalshami83@gmail.com}
}
\date{}

\begin{document}
\maketitle
\begin{abstract} In this paper, we study Levine equivalence in aura topological spaces. We show that the aura topology is determined by the reachability preorder induced by the scope function and is therefore an Alexandrov topology. We prove that the aura-Levine hull has the explicit representation $K_{\mathfrak a}(A)=\uparrow_{\mathfrak a}A=S_{\mathfrak a}^{\infty}(A)$, and hence two subsets are aura-Levine equivalent if and only if they have the same eventual forward spread. We also describe the quotient structure of the equivalence classes, characterize separation properties through the reachability preorder, classify scope functions that induce the same Levine equivalence, and examine preservation under aura-continuous mappings. For finite aura spaces, we provide an explicit procedure for deciding aura-Levine equivalence using the reflexive-transitive closure of the scope relation. \end{abstract}



\section{Introduction}
General topology has been substantially enriched by introducing generalized forms of open and closed sets and by modifying the families through which subsets of a space are observed. Classical examples include semi-open sets, generalized closed sets, nearly open sets, preopen sets, $\alpha$-continuous mappings, and $\beta$-open sets \cite{Levine1963,Levine1970,Njaastad1965,CrossleyHildebrand1972,Hamlett1975,Mashhour1982,ReillyVamanamurthy1985,AbdElMonsef1983}. Related approaches based on ideals, grills, primal structures, and associated closure operators have further enlarged the range of topological structures used to distinguish subsets \cite{JankovicHamlett1990,RoyMukherjee2007,Acharjee2025,AlOmariAlqahtani2023,Alghamdi2024,Alqahtani2024}. These developments show the continuing importance of generalized families of open sets in determining how subsets of a space are observed and distinguished. In this broader context, it is natural to ask when two subsets are indistinguishable with respect to the chosen family of open sets, a question formalized by Levine's equivalence relation.

Levine \cite{Levine1971} introduced an equivalence relation on subsets of a topological space by declaring two sets equivalent whenever they are contained in exactly the same open sets. More precisely, for subsets $A,B$ of a topological space $(X,\tau)$,
\[
A \equiv B \quad \Longleftrightarrow \quad A \subseteq U \iff B\subseteq U
\]
for every $U \in \tau$. This relation provides a natural way of identifying subsets that are indistinguishable from the point of view of the topology. Levine equivalence has subsequently been considered in several generalized settings, including soft topologies \cite{Elsayed2023}, intuitionistic fuzzy topologies \cite{ElsayedAlShami2025}, and supra topologies \cite{AlShamiAzzam2026}.

Aura topological spaces were introduced by A\c{c}{\i}kg\"oz
\cite{AcikgozAura2026}. An aura topological space is a triple
$(X,\tau,\mathfrak a)$, where $(X,\tau)$ is a topological space and
\[
\mathfrak a:X\longrightarrow \tau
\]
is a scope function satisfying
\[
x\in\mathfrak a(x)
\]
for every $x\in X$. The scope function determines an associated aura
topology $\tau_{\mathfrak a}$ together with natural interior- and
closure-type operators. A\c{c}{\i}kg\"oz proved, among other results,
that the iterated aura closure gives rise to a Kuratowski closure whose
associated topology coincides with $\tau_{\mathfrak a}$, and also
established the equivalence of the aura-$T_1$ and aura-$T_2$
separation axioms \cite{AcikgozAura2026}. Further developments of the
framework include compactness and connectedness, ideal-aura spaces,
and fuzzy and soft versions
\cite{AcikgozCompact2026,AcikgozIdeal2026,AcikgozFuzzyAura2026,
AcikgozSoftAura2026}

The purpose of the present paper is to study Levine equivalence in the setting of aura topological spaces and, in particular, to reveal the order-theoretic mechanism underlying the framework. We show that the scope function generates a canonical reachability preorder and that the aura topology is exactly the Alexandrov topology of its upward-closed subsets. This representation provides explicit descriptions of the aura-Levine hull, the iterated aura closure, the separation axioms, and the quotient by Levine equivalence, and permits these notions to be studied uniformly through reachability. For this purpose, we consider the relation
\[
xR_{\mathfrak a}y \quad \Longleftrightarrow \quad y \in \mathfrak a(x),
\]
and denote its reflexive-transitive closure by $\preca$. One of the main observations of the paper is that the aura topology is completely determined by this preorder. In fact,
\[
\tau_{\mathfrak a} = \Up(X,\preca),
\]
so every aura topology is an Alexandrov topology.

This representation makes it possible to describe aura-Levine equivalence explicitly. We show that the canonical hull of a subset $A$ is
\[
K_{\mathfrak a}(A) = \uparrow_{\mathfrak a}A = S_{\mathfrak a}^{\infty}(A),
\]
and consequently
\[
A \equiv_{\mathfrak a}B \quad \Longleftrightarrow \quad
K_{\mathfrak a}(A) = K_{\mathfrak a}(B)
\quad \Longleftrightarrow \quad S_{\mathfrak a}^{\infty}(A) = S_{\mathfrak a}^{\infty}(B).
\]
Thus, aura-Levine equivalence can be understood entirely in terms of eventual reachability induced by the scope function.

We further study the quotient determined by this equivalence and show that it is naturally related to the frame of aura-open sets. Separation properties are characterized through the reachability preorder, and scope refinement and scope functions producing the same Levine equivalence are investigated. We also study the behavior of aura-Levine equivalence under aura-continuous mappings. Finally, for finite aura spaces, we give an explicit procedure for deciding whether two subsets are aura-Levine equivalent by computing the reflexive-transitive closure of the scope relation.

The results provide a direct connection between Levine equivalence, aura topology, and the preorder generated by the scope function, and show that the reachability structure gives a convenient framework for studying observational equivalence in aura topological spaces.

The remainder of the paper is organized as follows. Section \ref{sec:2} recalls the basic notions of Levine equivalence and aura topological spaces. Section \ref{sex:3} develops the reachability preorder associated with the scope function and establishes the Alexandrov structure of the aura topology. Section \ref{sex:4} introduces aura-Levine equivalence and studies its canonical hull representation. Section \ref{sex:5} investigates the quotient structure and its frame and Heyting properties. Section \ref{sex:6} discusses point indistinguishability and separation properties. Section \ref{sex:7} studies scope refinement and classifies scope functions that induce the same Levine equivalence. Section \ref{sex:8} considers aura-continuous mappings and the induced behavior on equivalence classes. Section \ref{sex:9} presents the finite case together with an algorithm for deciding aura-Levine equivalence and a detailed example.

\section{Preliminaries} \label{sec:2}

Throughout, $X$ is nonempty and $\Pset(X)$ denotes its power set.  For a relation $R$ on $X$, $R^*$ denotes the reflexive--transitive closure of $R$.

\subsection{Levine equivalence}

\begin{definition}[Levine \cite{Levine1971}]
Let $(X,\sigma)$ be a topological space.  For $A,B\subseteq X$, write $A\equiv_\sigma B$ when
\[
 A\subseteq U\quad\Longleftrightarrow\quad B\subseteq U
 \qquad\text{for every }U\in\sigma.
\]
The \emph{Levine hull} of $A$ is
\[
 K_\sigma(A)=\bigcap\{U\in\sigma:A\subseteq U\}.
\]
\end{definition}

Levine showed that $A\equiv_\sigma B$ iff $K_\sigma(A)=K_\sigma(B)$, that $K_\sigma$ is a Kuratowski closure operator, and that $K_\sigma(A)$ is the largest subset equivalent to $A$ \cite{Levine1971}.  The unusual feature is that $K_\sigma$ is built from intersections of \emph{open} supersets rather than closed supersets.  In a general topology $K_\sigma(A)$ need not itself be open.  This distinction disappears in Alexandrov spaces, and that observation will be decisive below.

\subsection{Aura topological spaces}

\begin{definition}[A\c{c}{\i}kg\"oz \cite{AcikgozAura2026}]
Let $(X,\tau)$ be a topological space.  A map $\mathfrak a:X\to\tau$ is a \emph{scope function} if $x\in\mathfrak a(x)$ for every $x\in X$.  The triple $(X,\tau,\mathfrak a)$ is an \emph{aura topological space}.  For $A\subseteq X$, set
\[
 \cla(A)=\{x:\mathfrak a(x)\cap A\neq\varnothing\},
 \qquad
 \inta(A)=\{x\in A:\mathfrak a(x)\subseteq A\}.
\]
A set $U$ is \emph{$\mathfrak a$-open} if $\inta(U)=U$; the family of all such sets is denoted by $\taua$.
\end{definition}

The aura paper proves that $\taua$ is a topology with $\taua\subseteq\tau$, and that $\cla$ is extensive, monotone and finitely additive but need not be idempotent \cite{AcikgozAura2026}.  Its iterates are
\[
 \cla^0(A)=A,\qquad \cla^{n+1}(A)=\cla(\cla^n(A)),
 \qquad
 \clainf(A)=\bigcup_{n\ge0}\cla^n(A).
\]
A forward-spread operator is also natural in the aura setting:
\[
 \Sa(A)=\bigcup_{x\in A}\mathfrak a(x),\qquad
 \Sa^0(A)=A,\qquad \Sa^{n+1}(A)=\Sa(\Sa^n(A)),\qquad
 \Sainf(A)=\bigcup_{n\ge0}\Sa^n(A).
\]
The distinction between $\cla$ and $\Sa$ is directional: $\cla$ collects points whose scopes meet $A$, while $\Sa$ collects points lying in scopes emanating from $A$.

\begin{definition}[Scope relation and reachability]
For an aura space $(X,\tau,\mathfrak a)$ define
\[
 xR_{\mathfrak a}y \iff y\in\mathfrak a(x).
\]
Its reflexive--transitive closure is denoted by $\preca$.  For $A\subseteq X$ put
\[
 \upar A=\{y:\exists x\in A,\ x\preca y\},\qquad
 \downar A=\{x:\exists y\in A,\ x\preca y\}.
\]
\end{definition}

Because $x\in\mathfrak a(x)$, $R_{\mathfrak a}$ is reflexive, so $\preca$ is a preorder.  We shall call it the \emph{aura reachability preorder}.
\section{The hidden Alexandrov structure of aura topology} \label{sex:3}

The following theorem is the structural starting point of the paper.

\begin{theorem} \label{thm:alex}
For every aura space $(X,\tau,\mathfrak a)$,
\[
 \taua=\operatorname{Up}(X,\preca)
 =\{U\subseteq X: x\in U,\ x\preca y\Rightarrow y\in U\}.
\]
Consequently, $\taua$ is an Alexandrov topology, i.e. it is closed under arbitrary intersections.
\end{theorem}

\begin{proof}
We prove the equality $\taua = \Up(X,\preca)$ by showing both inclusions. First, let $U \in \taua$. By the definition of $\mathfrak a$-openness, $\inta(U) = U$. Thus, for every $x \in U$, $\mathfrak a(x) \subseteq U$. we recall that
\[
xR_a y \iff y \in \mathfrak a(x).
\]
Hence, if $x \in U$ and $xR_a y$, then $y \in \mathfrak a(x)\subseteq U$. Now suppose that $x\in U$ and $x \preca y$. Since $\preca$ is the reflexive-transitive closure of $R_a$, there exist points $x = x_0, x_1, \ldots, x_n = y$ such that $x_i R_a x_{i+1} \quad (i = 0, \ldots, n-1)$.
We begin with $x_0 = x \in U$, the $\mathfrak a$-openness of $U$ gives $x_1 \in \mathfrak a(x_0) \subseteq U$. Since $x_1 \in U$, we again have $a(x_1) \subseteq U$, and therefore $x_2 \in U$. Repeating this argument yields $x_0, x_1, \ldots, x_n \in U$. Thus, $x \in U,\, x \preca y \, \Longrightarrow  y \in U$, so $U$ is a $\preca$-upset.

Consequently,
\[
\taua \subseteq \Up(X,\preca).
\]
Conversely, suppose that $U \in \Up(X,\preca)$. We show that $U$ is $\mathfrak a$-open. Let $x \in U$. If $y \in \mathfrak a(x)$, then by the definition of $R_a$, $x R_a y$.
Since $R_a \subseteq \preca$, we have
\[
x \preca y.
\]
Because $U$ is a $\preca$-upset and $x \in U$, we obtain $y \in U$. Since this holds for every $y \in \mathfrak a(x)$, $\mathfrak a(x) \subseteq U$. Hence, $U \subseteq \inta(U)$. By definition, $\inta(U) \subseteq U$. Thus, $\inta(U) = U$, and therefore $U \in \taua$. Hence
\[
\Up(X,\preca) \subseteq \taua.
\]
Finally, we show that arbitrary intersections of $\preca$-upsets are $\preca$-upsets. Let $(U_i)_{i\in I}$ be a family of $\preca$-upsets. Suppose that
\[
x \in \bigcap_{i\in I}U_i \quad\text{and}\quad
x \preca y.
\]
Then $x \in U_i$ for every $i \in I$. Since each $U_i$ is an upset,
$ y \in U_i$ for every $i \in I$. Therefore
\[
y \in \bigcap_{i\in I}U_i.
\]
Thus $\bigcap_{i\in I}U_i$ is a $\preca$-upset.

Hence $\taua$ is closed under arbitrary intersections, and consequently $\taua$ is an Alexandrov topology.
\end{proof}

Theorem \ref{thm:alex} places aura topological spaces within the classical correspondence between preordered sets and Alexandrov topological spaces. In particular, a preorder determines an Alexandrov topology through its family of upward-closed sets, while the specialization preorder recovers the corresponding order-theoretic structure; see, for example, \cite{Alexandroff1937,Stong1966,Barmak2011}. The contribution here is the identification of the preorder generated by the scope function, namely $\preccurlyeq_{\mathfrak a}=R_{\mathfrak a}^{*}$, as the preorder whose upset topology is precisely the aura topology $\tau_{\mathfrak a}$

\begin{corollary} \label{cor:minnbhd}
For each $x\in X$, the smallest $\mathfrak a$-open neighborhood of $x$ is
\[
 N_{\mathfrak a}(x)=\upar\!\{x\}=\Sainf(\{x\}).
\]
For every $A\subseteq X$, the smallest $\mathfrak a$-open set containing $A$ is
\[
 \upar A=\Sainf(A)=\bigcup_{x\in A}N_{\mathfrak a}(x).
\]
\end{corollary}

\begin{proof}
By Theorem \ref{thm:alex}, the aura-open sets are exactly the $\preca$-upsets. Fix $x\in X$. The principal upset
\[
\upar\{x\} = \{y \in X:x \preca y\}
\]
is an aura-open set containing $x$. Let $U \in \taua$ with $x \in U$. Let $y \in \upar\{x\}$, then $x \preca y$, which implies $y \in U$. Hence, $\upar\{x\}\subseteq U$. Therefore, $\upar \{x\}$ is the smallest aura-open neighborhood of $x$, so
\[
N_{\mathfrak a}(x) = \upar \{x\}.
\]

By definition of $\preca$, a point $y$ belongs to $\upar\!\{x\}$ exactly when it is reachable from $x$ by finitely many $R_{\mathfrak a}$-steps.  This is precisely the condition that $y$ belong to $\Sainf(\{x\})$.  Thus
\[
 N_{\mathfrak a}(x)=\upar\!\{x\}=\Sainf(\{x\}).
\]

Now let $A \subseteq X$. Since an upset containing $A$ must contain the principal upset generated by every point of $A$, the smallest $\mathfrak a$-open set containing $A$ is
\[
\upar A = \bigcup_{x \in A}\upar\{x\}.
\]
Using the pointwise identity established above, we obtain $\upar A = \bigcup_{x\in A}\Sainf(\{x\}) = \Sainf(A)$. Hence
\[
\upar A = \Sainf(A) = \bigcup_{x\in A}N_{\mathfrak a}(x),
\]
which proves the result.
\end{proof}

The backward operator has an equally simple form.

\begin{proposition} \label{prop:backward}
For every $A \subseteq X$,
\[
\clainf(A) = \downar A.
\]
Moreover, $\clainf$ is exactly the ordinary topological closure operator of $(X,\taua)$.
\end{proposition}

\begin{proof}
We first show that, for every $n\geq0$, a point $x$ belongs to $\cla^n(A)$ if and only if there is an $R_{\mathfrak a}$-path of length at most $n$ beginning at $x$ and ending at a point of $A$.  For $n=0$, this is immediate from $\cla^0(A)=A$, since membership in $A$ corresponds to a path of length zero.

Assume the statement holds for some $n\geq0$.  By definition,
\[
 x\in\cla^{n+1}(A)
\]
if and only if $\mathfrak a(x)\cap\cla^n(A)\neq\varnothing$.  Thus there is some $y\in\cla^n(A)$ with $xR_{\mathfrak a}y$.  By the induction hypothesis, $y$ is joined to a point of $A$ by a path of length at most $n$, and adjoining the initial edge $xR_{\mathfrak a}y$ gives a path from $x$ to $A$ of length at most $n+1$.  For the converse, if the path has length $0$, then $x\in A\subseteq\cla^{n+1}(A)$; otherwise, reading the path from its first edge gives the required membership. This proves the claim by induction.

Taking the union over all $n$ therefore gives
\[
 x\in\clainf(A)
 \iff
 \text{$x$ reaches some $y\in A$ by finitely many $R_{\mathfrak a}$-steps}
 \iff
 x\in\downar A.
\]
Hence
\[
 \clainf(A)=\downar A.
\]

By Theorem \ref{thm:alex}, $\taua$ is the upset Alexandrov topology of $\preca$.  Its closed sets are therefore exactly the $\preca$-downsets.  The smallest downset containing $A$ is $\downar A$, so the ordinary topological closure of $A$ in $\taua$ is
\[
 \operatorname{cl}_{\taua}(A)=\downar A=\clainf(A).
\]
Thus $\clainf$ is exactly the closure operator of $(X,\taua)$.
\end{proof}

\begin{remark}
A\c{c}{\i}kg\"oz \cite{AcikgozAura2026} proved that the
Kuratowski closure obtained by iterating the aura closure induces precisely the aura topology $\tau_{\mathfrak a}$. Proposition \ref{prop:backward} provides an order-theoretic strengthening and an explicit description of this phenomenon. Namely,
\[
\clainf = \downar A = \operatorname{cl}_{\taua(A)}
\]
for every $A \subseteq X$.
Thus the resulting closure is completely determined by finite
reachability in the scope relation, and the equality of the induced topology with $\taua$ follows directly from the
reachability preorder.
\end{remark}

\begin{theorem} \label{thm:transitive}
The following are equivalent:
\begin{enumerate}[label=(\roman*)]
\item $\mathfrak a$ is transitive: $y \in\mathfrak a(x)$ implies $\mathfrak a(y) \subseteq \mathfrak a(x)$;
\item $R_{\mathfrak a}$ is transitive;
\item $R_{\mathfrak a} = \preca$;
\item $\cla$ is idempotent;
\item $\cla = \clainf$;
\item $\Sa$ is idempotent;
\item $\Sa = \Sainf$.
\end{enumerate}
\end{theorem}

\begin{proof}
(i) $\iff$ (ii) is immediate from $xR_{\mathfrak a}y \iff y \in \mathfrak a(x)$. 

(ii) $\iff$ (iii) holds since $R_{\mathfrak a}$ is reflexive. 

(ii)$ \iff$ (iv) assume that $R_{\mathfrak a}$ is transitive. If $x \in \cla^2(A)$, then there exists $y \in \cla(A)$ such that $x R_{\mathfrak a}y$. Since $y \in \cla(A)$, there exists $z \in A$ with $y R_{\mathfrak a}z$. By transitivity, $x R_{\mathfrak a}z$. and therefore $x \in \cla(A)$. Hence $\cla^2(A) \subseteq \cla(A)$. The reverse inclusion follows from the extensivity of $\cla$, and therefore
\[
\cla^2(A) = \cla(A).
\]
Thus $\cla$ is idempotent. Conversely, suppose that $R_{\mathfrak a}$ is not transitive. Then there exist $x,y,z \in X$ such that
\[
x R_{\mathfrak a}y,\quad y R_{\mathfrak a}z,\quad x \not R_{\mathfrak a}z.
\]
Equivalently,
\[
y \in \mathfrak a(x),\quad z \in \mathfrak a(y),\quad z \notin \mathfrak a(x).
\]
Since $yR_{\mathfrak a}z$, we have $y \in \cla(\{z\})$. Together with $x R_{\mathfrak a}y$, this gives $x \in \cla^2(\{z\})$. Also, $z \notin \mathfrak a(x)$ implies $x \notin \cla(\{z\})$. Hence
\[
x \in \cla^2(\{z\}) \setminus \cla(\{z\}),
\]
so $\cla$ is not idempotent.

(iv) $\iff$ (v) since $\cla$ is extensive, the iterates form an increasing sequence
\[
A \subseteq \cla(A) \subseteq \cla^2(A) \subseteq \cdots.
\]
If $\cla$ is idempotent, then $\cla^n(A)=\cla(A)$ for every $n\geq1$. Consequently,
\[
\clainf(A) = \cla(A).
\]
Conversely, if $\cla=\clainf$, then
\[
\cla^2(A)\subseteq\clainf(A)=\cla(A),
\]
Hence,
\[
\cla^2(A)=\cla(A).
\]

(ii) $\iff$ (vi) Suppose first that $R_{\mathfrak a}$ is transitive. If $z\in\Sa^2(A)$, then there exists $y \in \Sa(A)$ such that $y R_{\mathfrak a}z$. Since $y \in \Sa(A)$, there exists $x \in A$ such that $x R_{\mathfrak a}y$. By transitivity, $x R_{\mathfrak a}z$, and hence $z \in \Sa(A)$. Therefore $\Sa^2(A) = \Sa(A)$, so $\Sa$ is idempotent. Conversely, if transitivity fails, choose $x,y,z$ such that
\[
x R_{\mathfrak a}y,\quad y R_{\mathfrak a}z,\quad x \not R_{\mathfrak a}z.
\]
Then $z \in \Sa^2(\{x\})$, but $z \notin \Sa(\{x\})$, and therefore
\[
z \in \Sa^2(\{x\}) \setminus \Sa(\{x\}).
\]
Thus $\Sa$ is not idempotent.

(vi) $\iff$ (vii) since $\Sa$ is extensive, idempotence is equivalent to stabilization after one step. Hence $\Sa=\Sainf$ if and only if $\Sa$ is idempotent. 
\end{proof}
\section{Aura-Levine equivalence} \label{sex:4}

\begin{definition} \label{def:aura-levine}
For $A,B\subseteq X$, define
\[
A \eqa B \quad \Longleftrightarrow \quad \bigl(A \subseteq U \iff B \subseteq U \bigr) \quad\text{for every }U \in \taua.
\]
The corresponding canonical hull is
\[
\Ka(A) = \bigcap\{U \in \taua:A \subseteq U\}.
\]
\end{definition}

Thus $\eqa$ is Levine's relation computed in the topology generated by the scope function.  The next result is the central representation theorem.

\begin{theorem} \label{thm:main-equivalence}
For all $A,B \subseteq X$,
\[
\begin{aligned}
 A\eqa B
 &\Longleftrightarrow \Ka(A) = \Ka(B)\\
 &\Longleftrightarrow \upar A = \upar B\\
 &\Longleftrightarrow \Sainf(A) = \Sainf(B).
\end{aligned}
\]
Moreover,
\[
\Ka(A)=\upar A=\Sainf(A).
\]
\end{theorem}

\begin{proof}
Levine's general hull characterization gives
\[
 A \eqa B \iff \Ka(A) = \Ka(B).
\]
It remains to identify the hull explicitly in the aura topology. By definition,
\[
\Ka(A) = \bigcap \{U \in \taua: A \subseteq U\}.
\]
By Theorem \ref{thm:alex}, $\taua$ is Alexandrov, so arbitrary intersections of aura-open sets are aura-open.  Hence, $\Ka(A)$ is aura-open. By definition, we have $A \subseteq \Ka(A)$. Moreover, if $V \in \taua$ and $A \subseteq V$, then $\Ka(A) \subseteq V$. Thus $\Ka(A)$ is exactly the least aura-open superset of $A$.

By Corollary \ref{cor:minnbhd}, that least aura-open superset is
\[
\upar A = \Sainf(A).
\]
Therefore,
\[
\Ka(A) = \upar A = \Sainf(A).
\]
Hence,
\[
\Ka(A) = \Ka(B) \iff \upar A = \upar B \iff \Sainf(A) = \Sainf(B).
\]
\end{proof}

\begin{corollary} \label{cor:largest}
For every $A \subseteq X$, $\Ka(A)$ is the largest member of the aura-Levine equivalence class $[A]_{\mathfrak a}$. In particular,
\[
B \eqa A \quad \Longrightarrow \quad B \subseteq \Sainf(A).
\]
\end{corollary}

\begin{proof}
By Theorem \ref{thm:main-equivalence}, $A$ and $\Ka(A)$ have the same hull because
\[
\Ka(\Ka(A)) = \Ka(A).
\]
Hence, $A \eqa \Ka(A)$, so $\Ka(A) \in [A]_{\mathfrak a}$. Now let $B \eqa A$. By Theorem \ref{thm:main-equivalence},
\[
\Ka(B) = \Ka(A).
\]
Since $B \subseteq \Ka(B)$,
\[
B \subseteq \Ka(B) = \Ka(A).
\]
Thus every representative equivalent to $A$ is contained in $\Ka(A)$, proving that $\Ka(A)$ is the largest representative of the class. Finally, Theorem \ref{thm:main-equivalence} gives $\Ka(A) = \Sainf(A)$, and hence
\[
B \subseteq \Sainf(A).
\]
\end{proof}

The representation obtained in Theorem \ref{thm:main-equivalence} allows the algebraic properties of the aura-Levine hull to be derived directly from the Alexandrov structure of the aura topology. We now record these consequences, which will also provide the basis for the quotient structure studied in the next section.

\begin{proposition} \label{prop:hull-algebra}
The map $\Ka:\Pset(X) \to \Pset(X)$ satisfies
\begin{enumerate}[label=(\alph*)]
\item $\Ka(\varnothing) = \varnothing$ and $A \subseteq \Ka(A)$;
\item $A \subseteq B \Rightarrow \Ka(A) \subseteq \Ka(B)$;
\item $\Ka(\Ka(A)) = \Ka(A)$;
\item $\Ka \bigl(\bigcup_{i\in I}A_i\bigr) = \bigcup_{i\in I}\Ka(A_i)$ for every family $(A_i)_{i\in I}$;
\item $\operatorname{Fix}(\Ka) = \taua$.
\end{enumerate}
Thus $\Ka$ is a completely additive Kuratowski closure operator whose closed sets, in its own induced topology, are exactly the aura-open sets.
\end{proposition}

\begin{proof}

By Theorem \ref{thm:alex}, $\taua$ is an Alexandrov topology and is therefore closed under arbitrary intersections. Consequently, $\Ka(A)$ is itself aura-open for every $A \subseteq X$.

For (a), since $\varnothing \in \taua$, we obtain $\Ka(\varnothing) = \varnothing$. By definition, $A \subseteq \Ka(A)$.

For (b), suppose that $A \subseteq B$. Every aura-open set containing $B$ also contains $A$. Therefore,
\[
\Ka(A) = \bigcap \{U\in \taua:A\subseteq U\}
\subseteq \bigcap \{U \in \taua:B \subseteq U\} = \Ka(B).
\]
Thus $\Ka$ is monotone.

For (c), Since $A$ and $\Ka(A)$ are equivalent, their hulls are equal, and hence by Theorem \ref{thm:main-equivalence}
\[
\Ka(\Ka(A)) = \Ka(A).
\]

For (d), let $(A_i)_{i\in I}$ be an arbitrary family of subsets of $X$. Since
\[
A_i \subseteq \bigcup_{j\in I}A_j
\]
for every $i \in I$, monotonicity gives
\[
\Ka(A_i) \subseteq \Ka \left(\bigcup_{j \in I}A_j \right).
\]
Thus,
\[
\bigcup_{i\in I}\Ka(A_i) \subseteq \Ka \left(\bigcup_{i \in I}A_i\right).
\]

For the reverse inclusion, each $\Ka(A_i)$ is aura-open, and therefore their union $\bigcup_{i\in I} \Ka(A_i)$ is aura-open. Moreover, extensivity gives $A_i \subseteq\Ka(A_i)$ for every $i$, so
\[
\bigcup_{i\in I}A_i \subseteq \bigcup_{i\in I}\Ka(A_i).
\]
Thus $\bigcup_{i\in I}\Ka(A_i)$ is an aura-open superset of $\bigcup_{i \in I} A_i$. By the definition of $\Ka$,
\[
\Ka \left(\bigcup_{i \in I}A_i \right) \subseteq \bigcup_{i \in I} \Ka(A_i).
\]
Therefore,
\[
\Ka \left(\bigcup_{i \in I}A_i\right) = \bigcup_{i \in I}\Ka(A_i).
\]

Finally, for (e), suppose that $U \in \taua$. Then
\[
\Ka(U) = U.
\]
Thus every aura-open set is a fixed point of $\Ka$.

Conversely, suppose that $\Ka(A) = A$. Since $\Ka(A)$ is an arbitrary intersection of aura-open sets and $\taua$ is Alexandrov, we have $\Ka(A) \in \taua$. Hence
\[
\operatorname{Fix}(\Ka) = \taua.
\]
\end{proof}

\begin{proposition} \label{prop:dual}
For every $A \subseteq X$,
\[
\Ka(A) = \upar A, \quad \clainf(A) = \downar A.
\]
Hence $\Ka$ is the closure operator for the Alexandrov topology of $\preca^{\mathrm{op}}$, while $\clainf$ is the closure operator for the Alexandrov topology of $\preca$.
\end{proposition}

Ordinary intersection is not compatible with $\eqa$ in general as it can be seen in the next example.

\begin{example} \label{ex:intersection}
Let $X=\{p,q,r\}$ with discrete topology and
\[
\mathfrak a(p) = \{p,q\},\qquad \mathfrak a(q) = \{q,r\},\qquad \mathfrak a(r) = \{r\}.
\]
Then $p \preca q \preca r$.  Put $A = \{p\}$, $B = \{p,q\}$ and $C = \{q\}$.  Since
\[
\Ka(A) = X =\Ka(B),
\]
we have $A \eqa B$.  But $A \cap C = \varnothing$ while $B \cap C = \{q\}$, and
\[
\Ka(\varnothing) = \varnothing \neq \{q,r\} = \Ka(\{q\}).
\]
Thus $A \cap C \not \eqa B \cap C$.
\end{example}
\section{The quotient frame and Heyting structure} \label{sex:5}
The preceding results show that each aura-Levine equivalence class is completely determined by its canonical hull. Since the fixed points of $K_{\mathfrak a}$ are exactly the aura-open sets, it is natural to identify the quotient by aura-Levine equivalence with the lattice of aura-open sets.
This yields the following structural representation.

Let
\[
\LQ = \Pset(X)/{\eqa}
\]
The quotient is ordered by
\[
[A]_{\mathfrak a}\le [B]_{\mathfrak a} \quad \Longleftrightarrow \quad \Ka(A) \subseteq \Ka(B).
\]
This is well defined by Theorem \ref{thm:main-equivalence}. We use $[A]$ instead of $[A]_{\mathfrak a}$.

\begin{theorem} \label{thm:frameiso}
The map
\[
\Phi_{\mathfrak a}: \LQ \longrightarrow \taua, \qquad
\Phi_{\mathfrak a}([A]) = \Ka(A),
\]
is an order isomorphism. Consequently, $\LQ$ is a complete frame.
\end{theorem}

\begin{proof}
First, $\Phi_{\mathfrak a}$ is well defined.  If $[A] = [B]$, then $A \eqa B$, and by Theorem \ref{thm:main-equivalence},
\[
\Ka(A) = \Ka(B).
\]
Now if $\Phi_{\mathfrak a}([A]) = \Phi_{\mathfrak a}([B])$, then $\Ka(A)=\Ka(B)$, hence $A \eqa B$ and therefore $[A]=[B]$, proving the injectivity of $\Phi_{\mathfrak a}$.

To prove surjectivity, let $U \in\taua$.  By Proposition \ref{prop:hull-algebra}, aura-open sets are the fixed points of $\Ka$, so $\Ka(U) = U$. Hence, 
\[
\Phi_{\mathfrak a}([U]) = U.
\]

By the definition of the quotient order,
\[
[A] \leq[B] \iff \Ka(A) \subseteq\Ka(B) \iff \Phi_{\mathfrak a}([A])\subseteq\Phi_{\mathfrak a}([B]).
\]
Thus $\Phi_{\mathfrak a}$ both preserves and reflects order, and is therefore an order isomorphism.

Finally, by Theorem \ref{thm:alex}, $\taua$ is the family of upsets of the preorder $\preca$. Arbitrary unions and arbitrary intersections of upsets are again upsets, and set-theoretic intersection distributes over arbitrary union.  Hence $\taua$ is a complete frame. Consequently, $\LQ$ is a complete frame as well.
\end{proof}

\begin{corollary} \label{cor:operations}
For any family $([A_i])_{i \in I}$ in $\LQ$,
\[
\bigvee_{i \in I}[A_i] = \left[\bigcup_{i \in I}A_i\right],
\]
and
\[
\bigwedge_{i \in I}[A_i] = \left[\bigcap_{i\in I}\Ka(A_i)\right].
\]
\end{corollary}

\begin{proof}
Under the order isomorphism $\Phi_{\mathfrak a}$ from Theorem \ref{thm:frameiso}, joins and meets in $\LQ$ correspond respectively to unions and intersections in $\taua$.

For the join, complete additivity of $\Ka$ gives
\[
\Phi_{\mathfrak a} \left(\left[\bigcup_{i \in I}A_i \right]\right) = \Ka\left(\bigcup_{i\in I}A_i\right) = \bigcup_{i\in I}\Ka(A_i).
\]
The right-hand side is the join of the family $\Phi_{\mathfrak a}([A_i])$ in $\taua$. Therefore
\[
\bigvee_{i\in I}[A_i] = \left[\bigcup_{i\in I}A_i\right].
\]

For the meet, each $\Ka(A_i)$ is an upset, and arbitrary intersections of upsets remain upsets. Hence,
\[
U = \bigcap_{i\in I}\Ka(A_i) \in \taua.
\]
Thus,
\[
 \Phi_{\mathfrak a}([U]) = U =\bigcap_{i\in I}\Phi_{\mathfrak a}([A_i]),
\]
which proves
\[
 \bigwedge_{i \in I}[A_i] = \left[\bigcap_{i\in I}\Ka(A_i)\right].
\]
\end{proof}

Because every frame is a complete Heyting algebra, the quotient possesses an implication. 

\begin{theorem} \label{thm:heyting}
For $U,V \in\taua$, define
\[
U \Rightarrow_{\mathfrak a}V =\{x \in X:(\upar\{x\}) \cap U \subseteq V\}.
\]
Then $U \Rightarrow_{\mathfrak a} V \in \taua$ and it is the largest $W \in \taua$ satisfying $W \cap U \subseteq V$.  Therefore the implication in $\LQ$ is
\[
 [A] \Rightarrow[B] = \bigl[\Ka(A) \Rightarrow_{\mathfrak a}\Ka(B)\bigr].
\]
\end{theorem}

\begin{proof}
Let
\[
H = U \Rightarrow_{\mathfrak a}V = \{x \in X:(\upar \{x\}) \cap U \subseteq V\}.
\]
We need to show first that $H$ is aura-open. Suppose $x \in H$ and $x \preca y$. Let $z \in \upar\{y\}$, then $y \preca z$. By transitivity of $\preca$, $x \preca z$. Thus, $z \in \upar\{x\}$ and hence
\[
\upar \{y\} \subseteq \upar\{x\}.
\]
Therefore,
\[
(\upar \{y\}) \cap U \subseteq (\upar \{x\}) \cap U \subseteq V,
\]
so $y \in H$.  Thus $H$ is a $\preca$-upset, and by Theorem \ref{thm:alex}, $H \in \taua$.

Now, let $x \in H \cap U$. since $x \in \upar \{x\}$, 
\[
x \in (\upar \{x\}) \cap U \subseteq V,
\]
so $H \cap U \subseteq V$.

Finally, we show that the maximality of $H$. Let $W \in \taua$ such that $W \cap U \subseteq V$ and $x \in W$.  Since $W$ is an upset, 
\[
\upar \{x\} \subseteq W.
\]
Consequently,
\[
(\upar \{x\}) \cap U \subseteq W \cap U \subseteq V.
\]
Thus $x \in H$, and hence $W \subseteq H$. Therefore $H$ is the largest aura-open set such that $H \cap U \subseteq V$, which is the Heyting implication in the frame $\taua$.  Transporting the operation through $\Phi_{\mathfrak a}$ gives the stated implication formula in $\LQ$.
\end{proof}

\begin{definition}
Let $(X,\tau,\mathfrak a)$ be an aura topological space, and let $R_{\mathfrak a}$ be its scope relation. A nonempty subset $C \subseteq X$ is called a strongly connected component if for every $x,y \in C$,
\[
x \preca y \quad\text{and}\quad y\preca x,
\]
and $C$ is maximal with respect to this property.

Equivalently, if
\[
x \sim_{\preca}y \iff x \preca y \text{ and } y\preca x,
\]
then the strongly connected components are exactly the equivalence
classes of $\sim_{\preca}$.
\end{definition}

\begin{theorem} \label{thm:boolean}
The following are equivalent:
\begin{enumerate}[label=(\roman*)]
\item $\LQ$ is a Boolean algebra;
\item $\taua$ is closed under complements;
\item $\preca$ is symmetric;
\item $(X, \taua)$ is $R_0$;
\item $\taua$ is a partition topology consisting of arbitrary unions of strongly connected components.
\end{enumerate}
\end{theorem}

\begin{proof}
(i) $\iff$ (ii) by Theorem \ref{thm:frameiso}, $\LQ \cong \taua$. Hence,
\[
\LQ \text{ is Boolean} \iff \taua \text{ is Boolean} \iff
(\forall U \in \taua)\,(X \setminus U \in \taua).
\]

(ii) $\Longrightarrow$ (iii) let $x\preca y$. Suppose $y \not \preca x$. Then $y \in \upar \{y\}$ and $x\notin \upar \{y\}$. Since
$\upar\{y\}\in\taua$, $X \setminus \upar \{y\} \in \taua$.
Moreover, $x \in X \setminus\upar \{y\}$ and $y \notin X \setminus\upar\{y\}$. By Theorem \ref{thm:alex}, $X\setminus\upar\{y\}$ is a $\preca$-upset. Hence $x \preca y \text{ and } x \in X\setminus\upar\{y\}$ imply $y \in X \setminus\upar\{y\}$, a contradiction. Therefore
\[
x\preca y\Rightarrow y\preca x.
\]
so $\preca$ is symmetric. 
 
(iii) $\Longrightarrow$ (ii) Since $\preca$ is reflexive and transitive,
$ \preca$ is an equivalence relation. Let $U\in\taua$. By Theorem \ref{thm:alex}, $x \in U,\ x\preca y \Rightarrow y \in U$. If $x\in U$ and $ x \sim_{\preca} y$, then $x\preca y$, so $y\in U$. Hence
\[
U = \bigcup_{x\in U}[x]_{\preca}.
\]
Therefore
\[
X \setminus U = \bigcup_{x\notin U}[x]_{\preca},
\]
which is again a union of $\preca$-equivalence classes, hence a $\preca$-upset. Thus
\[
X \setminus U \in \taua.
\]

(iii) $\iff$ (iv) Recall that a topological space is $R_0$ if and only if its specialization preorder is symmetric. By Theorem \ref{thm:alex}, the specialization preorder of $(X,\taua)$ is $\preca$. Therefore,
\[
(X,\taua)\text{ is }R_0 \iff \bigl(x\preca y \Rightarrow y\preca x\bigr)
\]
for all $x,y\in X$.

(iii) $\Longrightarrow$ (v) if (iii) holds, then $\preca$ is an equivalence relation and
\[
\taua = \left\{\bigcup_{C \in \mathcal C}C: \mathcal C \subseteq X/{\sim_{\preca}} \right\}.
\]
Since the $\sim_{\preca}$-classes are exactly the strongly connected components,
\[
\taua = \left\{ \text{arbitrary unions of strongly connected components} \right\},
\]
which is the partition topology

(v) $\Longrightarrow$ (ii) if
\[
\taua = \left\{ \text{arbitrary unions of strongly connected components} \right\},
\]
then for every $U\in\taua$, $X \setminus U$ is again a union of strongly connected components. Hence
\[
X \setminus U \in \taua.
\]
\end{proof}
\section{Point indistinguishability, strongly connected components, and separation} \label{sex:6}

Throughout this section, we say that an aura topological space $(X,\tau,\mathfrak a)$ is aura-$T_i$ \cite{AcikgozAura2026} if the associated topological space $(X,\tau_{\mathfrak a})$ satisfies the separation axiom $T_i$.

\begin{definition}
For $x, y \in X$, write $x \approx_{\mathfrak a}y$ if $x$ and $y$ belong to exactly the same aura-open sets.  Equivalently, $\{x\} \eqa \{y\}$.
\end{definition}

\begin{theorem} \label{thm:scc}
For $x,y \in X$, the following are equivalent:
\begin{enumerate}[label=(\roman*)]
\item $x \approx_{\mathfrak a} y$;
\item $\{x\} \eqa \{y\}$;
\item $\upar \{x\} = \upar\{y\}$;
\item $x \preca y$ and $y \preca x$;
\item $x$ and $y$ lie in the same strongly connected component.
\end{enumerate}
\end{theorem}

\begin{proof}
(i) $\iff$ (ii) follows by definition. 

(ii) $\iff$ (iii) by Theorem \ref{thm:main-equivalence},
\[
\{x\} \eqa \{y\} \iff \upar \{x\} = \upar \{y\}.
\]

(iii) $\Longrightarrow$ (iv) since $\preca$ is reflexive, $y \in \upar \{y\} = \upar \{x\}$, which gives $x \preca y$.  Similarly, $ x \in\upar \{x\} = \upar \{y\}$, so $y \preca x$.

(iv) $\Longrightarrow$ (iii) suppose $x \preca y$ and $y \preca x$. If $z \in \upar \{x\}$, then $x\preca z$. From $y \preca x$ and transitivity we obtain $y \preca z$, so $z \in \upar \{y\}$. Thus $\upar \{x\} \subseteq\upar \{y\}$. Similarly, $\upar \{y\} \subseteq \upar \{x\}$. Hence,
\[
\upar\{x\} = \upar\{y\}.
\]

(iv) $\iff$ (v) $x \preca y$ and $y \preca x$ $\iff$ $[x] = [y]$.
\end{proof}

\begin{corollary} \label{cor:t0}
The following are equivalent:
\begin{enumerate}[label=(\roman*)]
\item $(X,\tau,\mathfrak a)$ is aura-$T_0$;
\item $\preca$ is antisymmetric;
\item every strongly connected component is a singleton;
\item $x \mapsto \Ka(\{x\})$ is injective.
\end{enumerate}
\end{corollary}

\begin{proof}

(i) $\iff$ (ii) $\iff$ (iii) follows by definition.

(ii) $\iff$ (iv) follows from Theorem \ref{thm:scc}, where
\[x \preca y \text{ and } y \preca x \iff \Ka(\{x\}) = \Ka(\{y\}).
\]
\end{proof}

A\c{c}{\i}kg\"oz \cite{AcikgozAura2026} established the equivalence of (i), (ii), (iii), and (v) in the next theorem. The next theorem extends that equivalence to the reachability preorder and ensures that an aura-$T_1$ space necessarily has discrete ambient topology..

\begin{theorem} \label{thm:t1t2}
For an aura space, the following are equivalent:
\begin{enumerate}[label=(\roman*)]
\item the space is aura-$T_1$;
\item the space is aura-$T_2$;
\item $\taua = \Pset(X)$;
\item $\preca$ is equality;
\item $\mathfrak a(x) = \{x\}$ for every $x\in X$.
\end{enumerate}
In particular, an aura-$T_1$ space necessarily has discrete ambient topology.
\end{theorem}

\begin{proof}
(ii) $\Longrightarrow$ (i) Clear.

(i) $\Longrightarrow$ (iii) By Theorem \ref{thm:alex}, $\taua$ is Alexandrov. Hence, for each $x \in X$,
\[
N_{\mathfrak a}(x) = \bigcap\{U \in \taua:x \in U\} \in\taua.
\]
By Corollary \ref{cor:minnbhd}, $N_{\mathfrak a}(x)$ is the least aura-open neighborhood of $x$. Let $y \in N_{\mathfrak a}(x)$. If $y \neq x$, then the $T_1$ property yields some $U \in \taua$ such that $x \in U$ and $y \notin U$. But
\[
N_{\mathfrak a}(x) = \bigcap\{V\in\taua:x\in V\} \subseteq U,
\]
which contradicts $y \in N_{\mathfrak a}(x)$. Hence, $N_{\mathfrak a}(x)  = \{x\}$. Thus, $\{x\} \in \taua$ for every $x \in X$, and hence
\[
\taua = \Pset(X).
\]

(iii) $\Longrightarrow$ (ii) $\taua = \Pset(X) \Rightarrow (X,\taua)\text{ is discrete}
\Rightarrow (X,\taua)\text{ is }T_2$.

(iii) $\iff$ (iv) By Corollary \ref{cor:minnbhd},
\[
N_{\mathfrak a}(x)=\upar \{x\} = \{y \in X:x \preca y\}.
\]
Therefore, $\taua = \Pset(X)$ if and only if $N_{\mathfrak a}(x) = \{x\}$ for every $x\in X$, which is equivalent to $\upar\{x\} = \{x\}$ for every $x \in X$. Equivalently,
\[
x \preca y \Rightarrow y = x.
\]
Since $\preca$ is reflexive, $x \preca x$ for every $x \in X$. Hence,
\[
x \preca y \iff x = y.
\]

(iv) $\Longrightarrow$ (v) Let $y \in \mathfrak a(x)$. By definition of $R_{\mathfrak a}$,
\[
y \in \mathfrak a(x) \iff x R_{\mathfrak a} y.
\]
Since $R_{\mathfrak a}\subseteq\preca$, we have
\[
x R_{\mathfrak a}y \Rightarrow x \preca y \Rightarrow x = y.
\]
Thus
\[
\mathfrak a(x) = \{x\}.
\]

(v) $\Longrightarrow$ (iv) By (v),
\[
x R_{\mathfrak a}y \iff y \in\mathfrak a(x) \iff y = x.
\]
Hence
\[
R_{\mathfrak a} = \{(x, x): x \in X\}.
\]
Thus,
\[
\preca = R_{\mathfrak a}
\]
\end{proof}

\begin{corollary} \label{cor:kolmogorov}
Let $Q=X/{\approx_{\mathfrak a}}$ be the set of strongly connected components, ordered by
\[
 [x]\le_Q[y]\quad\Longleftrightarrow\quad x\preca y.
\]
Then $(Q,\le_Q)$ is a poset, its upset topology is $T_0$, and the quotient map $\pi:X\to Q$ induces a frame isomorphism
\[
 \taua\cong\operatorname{Up}(Q,\le_Q),\qquad
 U\longmapsto\pi[U].
\]
Thus the $T_0$ reflection of an aura topology is the condensation poset of its scope digraph.
\end{corollary}

\begin{proof}
By Theorem \ref{thm:scc}, the equivalence classes of $\approx_{\mathfrak a}$ are exactly the strongly connected components, and the relation $\le_Q$ is therefore well defined and antisymmetric. Hence $(Q,\le_Q)$ is a poset. Every $U\in\taua$ is a $\preca$-upset and is saturated with respect to $\approx_{\mathfrak a}$, so $\pi[U]$ is an upset of $Q$ and $\pi^{-1}(\pi[U])=U$. Conversely, if $V$ is an upset of $Q$, then $\pi^{-1}(V)$ is a $\preca$-upset and hence belongs to $\taua$ by Theorem \ref{thm:alex}. Thus $U\mapsto\pi[U]$ and $V\mapsto\pi^{-1}(V)$ are inverse order isomorphisms, and therefore a frame isomorphism. Since the upset topology of a poset is $T_0$, the stated $T_0$ reflection follows.
\end{proof}
\section{Scope refinement and classification up to Levine observability} \label{sex:7}

\begin{definition}
For two scope functions $\mathfrak a_1,\mathfrak a_2$ on the same $(X,\tau)$, write
\[
\mathfrak a_2 \sqsubseteq \mathfrak a_1 \quad \Longleftrightarrow \quad \mathfrak a_2(x) \subseteq \mathfrak a_1(x) \quad(x \in X).
\]

\end{definition}

\begin{theorem} \label{thm:refinement}
If $\mathfrak a_2 \sqsubseteq\mathfrak a_1$, then
\begin{enumerate}[label=(\alph*)]
\item $R_{\mathfrak a_2} \subseteq R_{\mathfrak a_1}$ and $\preccurlyeq_{\mathfrak a_2} \subseteq \preccurlyeq_{\mathfrak a_1}$;
\item $\tau_{\mathfrak a_1} \subseteq\tau_{\mathfrak a_2}$;
\item $K_{\mathfrak a_2}(A)\subseteq K_{\mathfrak a_1}(A)$ for every $A \subseteq X$;
\item $\equiv_{\mathfrak a_2}\ \subseteq\ \equiv_{\mathfrak a_1}$ as equivalence relations on $\Pset(X)$.
\end{enumerate}
\end{theorem}

\begin{proof}
(a) Assume $\mathfrak a_2\sqsubseteq\mathfrak a_1$.  By definition, $\mathfrak a_2(x) \subseteq\mathfrak a_1(x)$ for every $x \in X$. Hence, if $xR_{\mathfrak a_2}y$, then $y \in \mathfrak a_2(x) \subseteq \mathfrak a_1(x)$, so $xR_{\mathfrak a_1}y$. Therefore, $R_{\mathfrak a_2} \subseteq R_{\mathfrak a_1}$. Let $a \preccurlyeq_{\mathfrak a_2} b$, then there exist $a = x_1, x_2, \ldots, x_n = b$ such that $x_i R_{\mathfrak a_2}x_{i+1}$ for $1 \leq i \leq n-1$. Since $R_{\mathfrak a_2} \subseteq R_{\mathfrak a_1}$, $\preccurlyeq_{\mathfrak a_2}\subseteq\preccurlyeq_{\mathfrak a_1}$.

(b) Let $U \in \tau_{\mathfrak a_1}$. By Theorem \ref{thm:alex}, $U$ is an upset for $\preccurlyeq_{\mathfrak a_1}$. Since $\preccurlyeq_{\mathfrak a_2} \subseteq \preccurlyeq_{\mathfrak a_1}$, $U$ is an upset for $\preccurlyeq_{\mathfrak a_2}$. Hence, $\tau_{\mathfrak a_1}\subseteq\tau_{\mathfrak a_2}$.

(c) let $y \in K_{\mathfrak a_2}(A)$, then $x \preccurlyeq_{\mathfrak a_2}y$ for some $x \in A$. By (a), $x \preccurlyeq_{\mathfrak a_1} y$, so $y \in K_{\mathfrak a_1}(A)$. Thus, $K_{\mathfrak a_2}(A)\subseteq K_{\mathfrak a_1}(A)$.

(d) Suppose $A \equiv_{\mathfrak a_2}B$. By (b), $\tau_{\mathfrak a_1} \subseteq \tau_{\mathfrak a_2}$. Therefore, for every $U \in \tau_{\mathfrak a_1}$,
\[
A \subseteq U \iff B \subseteq U,
\]
Hence $A \equiv_{\mathfrak a_1}B$, proving $\equiv_{\mathfrak a_2}\ \subseteq\ \equiv_{\mathfrak a_1}$. 
\end{proof}

\begin{theorem} \label{thm:scope-classification}
Let $\mathfrak a$ and $\mathfrak b$ be scope functions on the same set $X$. The following are equivalent:
\begin{enumerate}[label=(\roman*)]
\item $\equiv_{\mathfrak a} = \equiv_{\mathfrak b}$ on $\Pset(X)$;
\item $K_{\mathfrak a} = K_{\mathfrak b}$;
\item $S_{\mathfrak a}^{\infty}(A) = S_{\mathfrak b}^{\infty}(A)$ for every $A \subseteq X$;
\item $\preccurlyeq_{\mathfrak a} = \preccurlyeq_{\mathfrak b}$;
\item $\tau_{\mathfrak a} = \tau_{\mathfrak b}$;
\item $R_{\mathfrak a}^{*} = R_{\mathfrak b}^{*}$.
\end{enumerate}
\end{theorem}

\begin{proof}
(ii) $\iff$ (iii) By Theorem \ref{thm:main-equivalence}, for every $A \subseteq X$,
\[
K_{\mathfrak a}(A) = S_{\mathfrak a}^{\infty}(A),
\qquad K_{\mathfrak b}(A) = S_{\mathfrak b}^{\infty}(A).
\]

(ii) $\Longleftrightarrow$ (iv) By Theorem \ref{thm:main-equivalence},
\[
K_{\mathfrak a}(\{x\}) = \{y:x\preccurlyeq_{\mathfrak a}y\},
\qquad
K_{\mathfrak b}(\{x\}) = \{y:x \preccurlyeq_{\mathfrak b}y\}.
\]
Hence, if $K_{\mathfrak a} = K_{\mathfrak b}$, then
\[
x\preccurlyeq_{\mathfrak a}y \iff x \preccurlyeq_{\mathfrak b}y.
\]
Conversely, if $\preccurlyeq_{\mathfrak a}=\preccurlyeq_{\mathfrak b}$, then
\[
K_{\mathfrak a}(\{x\}) = K_{\mathfrak b}(\{x\})
\]
for every $x \in X$. By complete additivity,
\[
K_{\mathfrak a}(A) = \bigcup_{x\in A}K_{\mathfrak a}(\{x\}) =
\bigcup_{x\in A}K_{\mathfrak b}(\{x\}) = K_{\mathfrak b}(A)
\]
for every $A \subseteq X$.

(iv) $\Longleftrightarrow$ (v) By Theorem \ref{thm:alex},
\[
\tau_{\mathfrak a} = \Up(X, \preccurlyeq_{\mathfrak a}), \quad \tau_{\mathfrak b} = \Up(X, \preccurlyeq_{\mathfrak b}).
\]
Hence, if $\preccurlyeq_{\mathfrak a} = \preccurlyeq_{\mathfrak b}$, then $\tau_{\mathfrak a} = \tau_{\mathfrak b}$. Conversely, suppose that $\tau_{\mathfrak a} = \tau_{\mathfrak b}$. 
By Theorem~\ref{thm:alex}, these topologies have specialization preorders $\preccurlyeq_{\mathfrak a}$ and $\preccurlyeq_{\mathfrak b}$, respectively. Thus,

\[
x\preccurlyeq_{\mathfrak a}y \iff (\forall U \in \tau_{\mathfrak a}) \bigl(x \in U \Rightarrow y \in U\bigr),
\]
similarly,
\[
x \preccurlyeq_{\mathfrak b}y \iff (\forall U\in\tau_{\mathfrak b}) \bigl(x\in U\Rightarrow y \in U\bigr).
\]
Thus,
\[
x \preccurlyeq_{\mathfrak a}y \iff x\preccurlyeq_{\mathfrak b}y
\]
for all $x, y \in X$. Therefore,
\[
\preccurlyeq_{\mathfrak a}=\preccurlyeq_{\mathfrak b}.
\]

(vi) $\iff$ (iv) By definition, $\preccurlyeq_{\mathfrak a} = R_{\mathfrak a}^{*}$ and $\preccurlyeq_{\mathfrak b} = R_{\mathfrak b}^{*}$

(ii) $\Longrightarrow$ (i) By Theorem \ref{thm:main-equivalence},
\[
A \equiv_{\mathfrak a}B \iff K_{\mathfrak a}(A) = K_{\mathfrak a}(B) \iff K_{\mathfrak b}(A) = K_{\mathfrak b}(B) \iff A \equiv_{\mathfrak b}B.
\]

(i) $\Longrightarrow$ (ii) Assume (i) by Corollary \ref{cor:largest}, for every $x \in X$
\[
K_{\mathfrak a}(\{x\}) \equiv_{\mathfrak a} \{x\} \iff K_{\mathfrak b}(\{x\}) \equiv_{\mathfrak b} \{x\}
\]
Since $K_{\mathfrak a}(\{x\})$ is the largest set such that $K_{\mathfrak a}(\{x\}) \equiv_{\mathfrak a} \{x\}$, we must have $K_{\mathfrak b}(\{x\}) \subseteq K_{\mathfrak a}(\{x\})$. Similarly, $K_{\mathfrak a}(\{x\}) \subseteq K_{\mathfrak b}(\{x\})$. Thus, $K_{\mathfrak a}(\{x\}) = K_{\mathfrak b}(\{x\})$. By complete additivity,
\[
K_{\mathfrak a}(A) = \bigcup_{x\in A}K_{\mathfrak a}(\{x\}) = \bigcup_{x\in A}K_{\mathfrak b}(\{x\}) = K_{\mathfrak b}(A)
\]
for every $A\subseteq X$.
\end{proof}

\begin{example} \label{ex:same-reach}
Let $X=\{p,q,r\}$ carry the discrete topology. Define
\[
\begin{array}{lll}
\mathfrak a(p) = \{p,q\},&\mathfrak a(q) = \{q,r\},&\mathfrak a(r) = \{r\},\\
 \mathfrak b(p) = X,&\mathfrak b(q) = \{q,r\},&\mathfrak b(r) = \{r\}.
\end{array}
\]
The two scope functions are different, since
\[
r \in \mathfrak b(p) \qquad\text{but}\qquad r\notin \mathfrak a(p).
\]
However, under $\mathfrak a$, the point $r$ is still reachable from $p$ through
\[
pR_{\mathfrak a}q \qquad\text{and}\qquad qR_{\mathfrak a}r.
\]
Thus, in both cases,
\[
p \preccurlyeq q \preccurlyeq r,
\]
and consequently
\[
\preccurlyeq_{\mathfrak a} = \preccurlyeq_{\mathfrak b}.
\]
Hence, by Theorem \ref{thm:scope-classification},
\[
\tau_{\mathfrak a} = \tau_{\mathfrak b} \qquad\text{and}\qquad \equiv_{\mathfrak a} = \equiv_{\mathfrak b}.
\]
Therefore, although the two scope functions are different, they induce the same reachability preorder, the same aura topology, and the same aura-Levine equivalence.
\end{example}
\section{Aura-continuous maps and preservation of Levine equivalence} \label{sex:8}

A map $f:(X,\tau,\mathfrak a)\to(Y,\sigma,\mathfrak b)$ is said to be aura-continuous \cite{AcikgozAura2026} when it is continuous from $(X,\taua)$ to $(Y,\tau_{\mathfrak b})$ .

\begin{theorem} \label{thm:monotone}
A function $f:X \to Y$ is aura-continuous if and only if
\[
x \preca y \quad \Longrightarrow\quad f(x)\preccurlyeq_{\mathfrak b}f(y).
\]
\end{theorem}

\begin{proof}
By Theorem \ref{thm:alex}, 
\[
\taua = \Up(X, \preca), \qquad \tau_{\mathfrak b} = \Up(Y, \preccurlyeq_{\mathfrak b}).
\]

Assume that $f$ is aura-continuous, and let $x \preca y$. We want to show that $f(x) \preccurlyeq_{\mathfrak b}f(y)$, let $V \in \tau_{\mathfrak b}$ be any aura-open set containing $f(x)$. Since $f$ is aura-continuous, $f^{-1}(V) \in \taua$.
Since $x \in f^{-1}(V)$ and $x \preca y$, the upset property of $f^{-1}(V)$ yields $y \in f^{-1}(V)$.  Hence $f(y) \in V$. Thus every $\mathfrak b$-open set containing $f(x)$ also contains $f(y)$, which means $f(x)\preccurlyeq_{\mathfrak b}f(y)$.

Conversely, suppose $f$ is monotone with respect to preorders. Let $V \in\tau_{\mathfrak b}$. We show that $f^{-1}(V)$ is a $\preca$-upset.  If $x \in f^{-1}(V)$ and $x \preca y$, then $f(x) \in V$ and monotonicity gives $f(x) \preccurlyeq_{\mathfrak b}f(y)$. Since $V$ is a $\preccurlyeq_{\mathfrak b}$-upset, $f(y) \in V$, so $y \in f^{-1}(V)$. Hence $f^{-1}(V) \in \taua$ by Theorem \ref{thm:alex}. Therefore $f$ is aura-continuous.
\end{proof}

\begin{theorem} \label{thm:image-preserve}
Let $f:(X,\tau,\mathfrak a) \to(Y,\sigma,\mathfrak b)$ be aura-continuous. If $A \eqa B$, then
\[
f(A) \equiv_{\mathfrak b} f(B).
\]
Hence there is a well-defined map
\[
\overline f:\LQ\to \mathsf L_{\mathfrak b}(Y), \qquad [A]_{\mathfrak a} \longmapsto[f(A)]_{\mathfrak b}.
\]
\end{theorem}

\begin{proof}
Let $V \in \tau_{\mathfrak b}$. For every subset $A \subseteq X$,
\[
f(A)\subseteq V \iff A \subseteq f^{-1}(V).
\]
Because $f$ is aura-continuous, $f^{-1}(V)\in\taua$. If $A\eqa B$, then
\[
A \subseteq f^{-1}(V) \iff B\subseteq f^{-1}(V).
\]
\[
B \subseteq f^{-1}(V) \iff f(B)\subseteq V.
\]
Thus, for every $V \in \tau_{\mathfrak b}$,
\[
f(A) \subseteq V \iff f(B) \subseteq V.
\]
By definition, that is 
\[
f(A) \equiv_{\mathfrak b}f(B).
\]
Hence $[f(A)]_{\mathfrak b}$ depends only on the class $[A]_{\mathfrak a}$, so $\overline f$ is well defined.
\end{proof}

\section{Finite aura spaces and computation} \label{sex:9}

Finite topological spaces are naturally encoded by preorders and, after $T_0$ reduction, by posets \cite{Stong1966,McCord1966,Barmak2011}.  Aura spaces make this encoding completely explicit through the scope graph.

Let $|X| = n$ and enumerate $X = \{x_1,\dots,x_n\}$. Write $M_{\mathfrak a}$ for the Boolean adjacency matrix
\[
(M_{\mathfrak a})_{ij} = 1\iff x_j\in\mathfrak a(x_i).
\]
Let $T_{\mathfrak a}$ denote its Boolean reflexive-transitive closure.

The next algorithm computes the aura-Levine equivalence.
\begin{algorithm} \label{alg:finite}
For finite $X$:
\begin{enumerate}
\item Build $M_{\mathfrak a}$ from the scope table.
\item Compute $T_{\mathfrak a}$ by Warshall's algorithm in $O(n^3)$ time \cite{Warshall1962}.
\item For $A \subseteq X$, compute the Boolean vector
\[
h_A = \bigvee_{x_i \in A}\operatorname{row}_i(T_{\mathfrak a}).
\]
Here, when $A=\varnothing$, the empty Boolean join is understood to be the zero vector. Then $h_A$ is the characteristic vector of $\Ka(A)$.
\item Decide $A\eqa B$ by testing $h_A = h_B$.
\end{enumerate}
\end{algorithm}

Note that $(T_{\mathfrak a})_{ij}=1$ if and only if
$x_i\preccurlyeq_{\mathfrak a}x_j$. Hence the $i$th row of
$T_{\mathfrak a}$ is the characteristic vector of
$K_{\mathfrak a}(\{x_i\})$. By complete additivity,
$h_A$ is therefore the characteristic vector of $K_{\mathfrak a}(A)$, and Theorem \ref{thm:main-equivalence} gives
\[
A \equiv_{\mathfrak a}B \quad \Longleftrightarrow \quad h_A = h_B.
\]
After $T_{\mathfrak a}$ has been computed once, each equivalence test therefore reduces to forming the relevant Boolean row joins and comparing the resulting $n$-bit vectors.

The next example illustrates how the algorithm works.
\begin{example} \label{ex:algorithm}

Let $X = \{p,q,r,s\}$ carry the discrete topology, and define the scope function $\mathfrak a$ by
\[
\mathfrak a(p) = \{p,q\},\quad
\mathfrak a(q) = \{q,r\},\quad
\mathfrak a(r) = \{r,s\},\quad
\mathfrak a(s) = \{s\}.
\]

We apply Algorithm \ref{alg:finite} step by step.

\noindent \textbf{Step 1. Construct the scope matrix.}
Using the ordering $p,q,r,s$, the Boolean matrix $M_{\mathfrak a}$ is defined by
\[
(M_{\mathfrak a})_{ij} = 1 \iff x_j \in \mathfrak a(x_i).
\]
Hence,
\[
M_{\mathfrak a} =
\begin{pmatrix}
1&1&0&0\\
0&1&1&0\\
0&0&1&1\\
0&0&0&1
\end{pmatrix}.
\]
For instance, the first row records $\mathfrak a(p)=\{p,q\}$, while the second row records $\mathfrak a(q)=\{q,r\}$.

\noindent \textbf{Step 2. Compute the reflexive--transitive closure.}
We have $r \notin \mathfrak a(p)$, but $q \in \mathfrak a(p)$ and $r \in \mathfrak a(q)$. Thus, $p\preca r$. Similarly, $p \preca s$ and $q \preca s$. Thus, the Boolean reflexive-transitive closure of $M_{\mathfrak a}$ is
\[
T_{\mathfrak a} =
\begin{pmatrix}
1&1&1&1\\
0&1&1&1\\
0&0&1&1\\
0&0&0&1
\end{pmatrix}.
\]

Each row of $T_{\mathfrak a}$ corresponds to a singleton hull. Therefore,
\[
K_{\mathfrak a}(\{p\}) = \{p,q,r,s\},
\]
\[
K_{\mathfrak a}(\{q\}) = \{q,r,s\},
\]
\[
K_{\mathfrak a}(\{r\}) = \{r,s\},
\quad K_{\mathfrak a}(\{s\}) = \{s\}.
\]

\noindent \textbf{Step 3. Compute the hull vector of a subset.}
Let $A = \{q,s\}$. The rows corresponding to $q$ and $s$ are
\[
\operatorname{row}_{q}(T_{\mathfrak a}) = (0,1,1,1)
\]
and
\[
\operatorname{row}_{s}(T_{\mathfrak a}) = (0,0,0,1).
\]
Taking their coordinatewise Boolean join gives
\[
h_A = (0,1,1,1) \vee(0,0,0,1) = (0,1,1,1).
\]
Hence
\[
K_{\mathfrak a}(A) = \{q,r,s\}.
\]

Now let $B = \{q\}$. Then
\[
h_B = \operatorname{row}_{q}(T_{\mathfrak a}) = (0,1,1,1),
\]
and therefore
\[
K_{\mathfrak a}(B)=\{q,r,s\}.
\]

\noindent \textbf{Step 4. Decide aura-Levine equivalence.}
Since $h_A = h_B$, Algorithm~\ref{alg:finite} yields
\[
A\equiv_{\mathfrak a}B.
\]

This illustrates that two subsets need not be equal in order to be aura-Levine equivalent. 

For comparison, let $C = \{r\}$. Then $h_C = (0,0,1,1)$, so
\[
h_C \neq h_B.
\]
Hence,
\[
\{r\} \not \equiv_{\mathfrak a}\{q\}.
\]

Therefore, after $T_{\mathfrak a}$ has been computed once, deciding whether two subsets are aura-Levine equivalent reduces simply to forming the Boolean joins of the appropriate rows and comparing the resulting vectors.
\end{example}


\subsection{Minimal representatives and the size of an equivalence class}
\label{subsec:minimal-representatives}

The preceding algorithm decides whether two subsets have the same
Levine hull. A related question is how small a representative of a given
equivalence class can be, and how many representatives that class contains.
The familiar description of upsets of finite posets by their minimal elements
\cite{Stong1966,Barmak2011} yields precise answers after passing to the
strongly connected components of the scope relation.

For $U\in\taua$, let $\mathcal M(U)$ denote the family of those strongly
connected components contained in $U$ whose images are minimal elements of
the induced subposet $\pi[U]\subseteq Q$, where $Q$ and $\pi:X\to Q$ are as
in Corollary~\ref{cor:kolmogorov}.

\begin{theorem}\label{thm:minimal-representatives}
Let $X$ be finite, $A\subseteq X$, and $U=\Ka(A)$. Then for every $B\subseteq X$,
\[
 B\eqa A\quad\Longleftrightarrow\quad
 B\subseteq U\ \text{and}\ B\cap C\ne\varnothing
 \ \text{for every }C\in\mathcal M(U).
\]
Moreover, $B$ is inclusion-minimal in $[A]_{\mathfrak a}$ if and only if
it contains exactly one point from each $C\in\mathcal M(U)$ and no other
points. Consequently,
\[
 \min\{|B|:B\eqa A\}=|\mathcal M(U)|,
\]
and the number of representatives of the class is
\[
 \bigl|[A]_{\mathfrak a}\bigr|
 =2^{\,|U|-\sum_{C\in\mathcal M(U)}|C|}
   \prod_{C\in\mathcal M(U)}(2^{|C|}-1).
\]
In particular, the number of minimum-cardinality representatives is
$\prod_{C\in\mathcal M(U)}|C|$.
\end{theorem}

\begin{proof}
Suppose first that $B\eqa A$. By Theorem~\ref{thm:main-equivalence},
$\Ka(B)=U$, and extensivity gives $B\subseteq U$.
Fix $C\in\mathcal M(U)$ and $x\in C$. Since $x\in U=\upar B$,
there exists $b\in B$ with $b\preca x$. Both $b$ and $x$ belong to $U$;
hence $\pi(b)\le_Q\pi(x)$ in the induced subposet $\pi[U]$.
Minimality of $\pi(x)$ forces $\pi(b)=\pi(x)$, so $b\in C$.
Thus $B\cap C\ne\varnothing$.

Conversely, suppose $B\subseteq U$ and $B\cap C\ne\varnothing$ for
all $C\in\mathcal M(U)$. Since $U$ is an upset,
$\Ka(B)=\upar B\subseteq U$. For each $x\in U$, finiteness of $Q$
provides a minimal element $C\in\mathcal M(U)$ with
$C\le_Q\pi(x)$. Choose $b\in B\cap C$. Then $b\preca x$,
so $x\in\Ka(B)$. It follows that $U\subseteq\Ka(B)$ and hence
$\Ka(B)=U$, as required.

The characterization shows that an inclusion-minimal representative
must select precisely one point of each component in $\mathcal M(U)$,
with no additional points. Such a selection is indeed minimal and has
cardinality $|\mathcal M(U)|$. There are
$\prod_{C\in\mathcal M(U)}|C|$ such selections. To count every
representative, choose an arbitrary nonempty subset of each minimal
component $C$ (giving $2^{|C|}-1$ choices) and an arbitrary subset
of the remaining $|U|-\sum_{C\in\mathcal M(U)}|C|$ points.
The stated formula follows. The argument also covers $U=\varnothing$,
using the usual conventions for empty sums and products.
\end{proof}

\begin{corollary}\label{cor:unique-minimal-representatives}
For a finite aura space, every aura-Levine equivalence class has a
unique inclusion-minimal representative if and only if the space is aura-$T_0$.
\end{corollary}

\begin{proof}
If the space is aura-$T_0$, every strongly connected component is a
singleton by Corollary~\ref{cor:t0}. Theorem~\ref{thm:minimal-representatives}
then gives exactly one inclusion-minimal representative for each class.
Conversely, if the space is not aura-$T_0$, there is a component
$C$ containing two distinct points $x$ and $y$. For
$U=\Ka(\{x\})=\Ka(\{y\})$, the component $C$ is the unique
minimal component of $\pi[U]$. Hence $\{x\}$ and $\{y\}$ are distinct
inclusion-minimal representatives of the same class.
\end{proof}

\begin{example}
Let $X=\{p,q,r\}$ have the discrete ambient topology and put
$\mathfrak a(p)=\{p,q\}$, $\mathfrak a(q)=\{p,q,r\}$, and
$\mathfrak a(r)=\{r\}$. The strongly connected components are
$\{p,q\}$ and $\{r\}$, with $\{p,q\}<_Q\{r\}$.
For $A=\{p\}$, we have $U=\Ka(A)=X$ and
$\mathcal M(U)=\{\{p,q\}\}$. Thus a set is equivalent to $A$
precisely when it contains $p$ or $q$; there are
$2^{3-2}(2^2-1)=6$ such sets. Its two inclusion-minimal
representatives are $\{p\}$ and $\{q\}$.
\end{example}

\section{Conclusion}
In this paper, we studied Levine equivalence in aura topological spaces through the reachability structure induced by the scope function. We showed that the aura topology is an Alexandrov topology determined by the reachability preorder and that the aura-Levine hull of a subset coincides with its eventual forward spread. Consequently, two subsets are aura-Levine equivalent precisely when they generate the same eventual spread.

We also investigated the algebraic and order-theoretic structure of the resulting equivalence classes, including the quotient frame and its Heyting structure, and characterized several separation properties in terms of the reachability preorder. In addition, we classified scope functions that induce the same aura-Levine equivalence, studied the behavior of this equivalence under aura-continuous mappings, and provided a finite procedure for deciding aura-Levine equivalence by means of the reflexive--transitive closure of the scope relation.

These results show that reachability provides a unified framework for understanding both the topology and the observational equivalence generated by a scope function. Further work may investigate analogous equivalence structures in fuzzy, soft, and other generalized forms of aura topological spaces.

\section*{Declarations}
\noindent\textbf{Conflict of interest.} The authors declare no conflict of interest.\\

\noindent\textbf{Data availability.} No data were generated or analyzed in this theoretical study.\\

\noindent\textbf{Funding.} No external funding is declared.\\

\noindent\textbf{Author contributions.} The authors conceived the study, developed the results, wrote the manuscript, and approved the final version.

\end{document}